\documentclass[11pt,letterpaper,reqno]{amsart}

\usepackage{tikz-cd}
\usepackage[T1]{fontenc}
\usepackage{amsmath,amssymb,amsfonts,amsthm}
\usepackage{mathtools}
\usepackage{aliascnt}
\usepackage{microtype}
\usepackage{enumitem}
\usepackage{booktabs}
\usepackage{xcolor}
\usepackage{doi}
\usepackage{tikz}
\usetikzlibrary{positioning,arrows.meta}
\usepackage{hyperref}
\usepackage{bookmark}
\usepackage[capitalize,noabbrev]{cleveref}

\newtheorem{theorem}{Theorem}[subsection]

\newaliascnt{lemma}{theorem}
\newtheorem{lemma}[lemma]{Lemma}
\aliascntresetthe{lemma}

\newaliascnt{proposition}{theorem}
\newtheorem{proposition}[proposition]{Proposition}
\aliascntresetthe{proposition}

\newaliascnt{corollary}{theorem}

\aliascntresetthe{corollary}

\newaliascnt{problem}{theorem}
\newtheorem{problem}[problem]{Problem}
\aliascntresetthe{problem}

\theoremstyle{definition}
\newaliascnt{definition}{theorem}

\aliascntresetthe{definition}

\newaliascnt{example}{theorem}

\aliascntresetthe{example}

\newaliascnt{remark}{theorem}

\aliascntresetthe{remark}

\crefname{theorem}{Theorem}{Theorems}
\Crefname{theorem}{Theorem}{Theorems}
\crefname{lemma}{Lemma}{Lemmas}
\Crefname{lemma}{Lemma}{Lemmas}
\crefname{proposition}{Proposition}{Propositions}
\Crefname{proposition}{Proposition}{Propositions}
\crefname{corollary}{Corollary}{Corollaries}
\Crefname{corollary}{Corollary}{Corollaries}
\crefname{problem}{Problem}{Problems}
\Crefname{problem}{Problem}{Problems}

\newcommand{\D}{\mathbb D}
\newcommand{\T}{\mathbb T}
\newcommand{\C}{\mathbb C}
\newcommand{\R}{\mathbb R}
\newcommand{\Ftwo}{\mathbb F_2}
\newcommand{\SpecS}{\operatorname{Spec}_{\mathrm{St}}}
\newcommand{\Aut}{\operatorname{Aut}}

\newcommand{\id}{\mathrm{id}}
\newcommand{\dd}{\,\mathrm{d}}
\newcommand{\iu}{\mathrm{i}}

\begin{document}

\title[Steklov-Isospectral Plane Domains]
{Strictly Convex Steklov-Isospectral Plane Domains}

\author[T.~Hu]{Tao Hu}
\address{
Department of Information, Risk, and Operations Management,
McCombs School of Business,
The University of Texas at Austin,
Austin, TX 78712, USA
}
\email{tao\_hu@utexas.edu}

\author[J.~Shi]{Jiachen Shi}
\address{School of Science, China University of Mining and Technology-Beijing, Beijing 100083, P.~R.~China}
\email{anssjc@outlook.com}

\author[Q.~Tang]{Quanyu Tang}
\address{School of Mathematical Sciences, University of Science and Technology of China, Hefei 230026, P.~R.~China}
\email{tangquanyu827@gmail.com}

\begin{abstract}
We construct pairs of noncongruent bounded Euclidean plane domains with identical Steklov spectra, including multiplicities. This gives a negative answer to the planar Steklov analogue of Kac's question ``Can one hear the shape of a drum?'' The domains are simply connected and strictly convex, have real-analytic boundaries, and may be chosen arbitrarily close to a disk in the $C^\infty$ topology.
\end{abstract}

\subjclass[2020]{Primary 58J53; Secondary 58J50, 30C35}

\keywords{Steklov spectrum, isospectral domains, Sunada method,
orbifolds, finite Blaschke products, conformal mappings}

\maketitle

\section{Introduction}

Kac's celebrated question ``Can one hear the shape of a drum?'' is a
foundational problem in inverse spectral geometry \cite{Kac1966}. Gordon, Webb, and Wolpert answered Kac's question negatively by
constructing a pair of nonisometric simply connected plane domains
having both the same Dirichlet spectrum and the same Neumann spectrum,
using Riemannian orbifold coverings
\cite{GordonWebbWolpert1992}. The Steklov problem is of a different
nature: the spectral parameter occurs in the boundary condition rather
than in the interior equation. Throughout the paper, by a domain we mean a connected open set. For a bounded connected plane domain
$\Omega$ with smooth boundary and outward unit normal $\nu$, the
Steklov eigenvalues form the sequence
\[
0=\sigma_0(\Omega)<\sigma_1(\Omega)\leq \sigma_2(\Omega)
\leq\cdots\nearrow\infty
\]
for which
\[
\begin{cases}
\Delta u = 0 & \text{in } \Omega \\ \partial_\nu u = \sigma u & \text{on } \partial\Omega
\end{cases}
\]
has a nonzero solution. Equivalently, they are the eigenvalues of the Dirichlet-to-Neumann
operator, which maps a boundary function to the normal derivative of
its harmonic extension. We write $\SpecS(\Omega)$ for the Steklov
spectrum, counted with multiplicity. General accounts of the problem and its spectral
geometry may be found in
\cite{ColboisGirouardGordonSher2024,GirouardPolterovich2017}.

Conformal mapping methods have long played a central role in the
planar Steklov problem \cite{Dittmar1988}.  Jollivet and Sharafutdinov reformulated the simply connected inverse problem as that of recovering a positive smooth boundary weight on the unit circle, up to the natural action of the conformal automorphism group of the disk \cite{JollivetSharafutdinov2014}.  Nevertheless, the
inverse Steklov problem for planar domains has exhibited substantial
rigidity.  Edward proved that the disk is spectrally determined among smooth bounded simply connected planar domains \cite{Edward1993}. Girouard, Parnovski, Polterovich, and Sher subsequently showed that the Steklov spectrum determines the number and lengths of the boundary components of every smooth compact Riemannian surface with boundary \cite{GirouardParnovskiPolterovichSher2014}. Together with Edward's result, this implies that the disk is spectrally determined among all smooth bounded planar domains.

Jollivet and Sharafutdinov established $C^\infty$-precompactness, after suitable normalization of the conformal parametrizations and modulo Euclidean isometries, for Steklov-isospectral families of smooth bounded simply connected, possibly multisheet, planar domains \cite{JollivetSharafutdinov2018}. The Steklov spectrum determines the side lengths of a rectangle, and every cube is spectrally determined within the class of cuboids~\cite{GirouardLagacePolterovichSavo2019}. For curvilinear polygons with all interior angles in $(0,\pi)$, sharp high-frequency asymptotics and corresponding inverse results under genericity assumptions were established in~\cite{KrymskiLevitinParnovskiPolterovichSher2021,LevitinParnovskiPolterovichSher2022}. More recently, spectral finiteness, meaning finiteness of the isospectral
set up to congruence, was proved for almost all convex polygons, together with spectral determination results for almost all triangles and further inverse spectral results for several special classes of quadrilaterals and regular polygons
\cite{DrydenGordonMorenoRowlettVillegas2025,
DrydenGordonMorenoRowlettVillegas2026}. Each circular annulus has
also been shown to be uniquely determined by its Steklov spectrum
among all planar domains with smooth boundary \cite{JinWang2026}.
Beyond the smooth setting, the Steklov Weyl law remains valid for
surfaces with Lipschitz, and even certain rougher, boundaries
\cite{KarpukhinLagacePolterovich2023}.

There is nevertheless considerable flexibility in limiting constructions. Boundary homogenization can produce domains that remain geometrically close to a prescribed limit while their normalized Steklov eigenvalues converge, for each fixed index, to those of a geometrically different target~\cite{BucurNahon2021,KarpukhinLagace2024}. Exact isospectral constructions are known for Riemannian manifolds and orbifolds with boundary. In particular, Gordon, Herbrich, and Webb adapted Sunada's method \cite{Sunada1985} and the torus-action method to the Steklov setting, obtaining Steklov-isospectral flat surfaces with boundary and planar domains with isospectral sloshing (mixed Steklov--Neumann) problems \cite{GordonHerbrichWebb2021}. In the familiar planar Sunada constructions, however, reflector strata in the quotient orbifolds become portions of the boundary of the underlying planar domains, on which the reflection-invariant and reflection-anti-invariant sectors impose Neumann and Dirichlet conditions, respectively. Thus these constructions yield mixed Steklov--Neumann or Steklov--Dirichlet problems rather than the pure Steklov problem on the entire boundary. This obstruction helps explain why the following question remained open. It was posed explicitly by Girouard and Polterovich \cite[Open Problem~6]{GirouardPolterovich2017}, who expressed the expectation that the answer would be negative, and was later restated in \cite[Open Question~9.11]{ColboisGirouardGordonSher2024}. It has continued to be recorded as open in the literature on inverse Steklov problems for curvilinear and convex polygonal domains \cite{DrydenGordonMorenoRowlettVillegas2025,DrydenGordonMorenoRowlettVillegas2026,KrymskiLevitinParnovskiPolterovichSher2021}.

\subsection{The planar inverse problem and main result}

\begin{problem}
\label{prob:plane-domains}
Do there exist two bounded, connected, noncongruent Euclidean plane
domains with identical Steklov spectra, including multiplicities?
\end{problem}

Contrary to the expectation expressed in
\cite[p.~344]{GirouardPolterovich2017}, our main result answers
\cref{prob:plane-domains} affirmatively and shows that neither
corners nor loss of convexity are needed.  We
identify $\R^2$ with $\C$, write $\iu=\sqrt{-1}$, put
$\D=\{z\in\C:|z|<1\}$ and $\T=\partial\D$, and use congruence to mean
equivalence under a Euclidean rigid motion. For each integer $m\geq0$, we use
\[
\|f\|_{C^m(\overline{\D})}
:=\max_{|\alpha|\leq m}\sup_{z\in\overline{\D}}
|\partial^\alpha f(z)|,
\qquad
\|g\|_{C^m(\T)}
:=\max_{0\leq k\leq m}\sup_{t\in\R}
\left|\frac{d^k}{dt^k}g(e^{\iu t})\right|.
\]
For $Z\in\{\overline{\D},\T\}$, convergence in $C^\infty(Z)$ means
convergence in the $C^m(Z)$ norm for every $m\geq0$.

\begin{theorem}
\label{thm:main}
There exist $\varepsilon_0>0$ and a family
$\{\Omega_{i,\varepsilon}\}_{i\in\{P,L\},\,0<\varepsilon<\varepsilon_0}$
of domains in $\R^2$ such that the following properties hold.
\begin{enumerate}[label=\textup{(\roman*)}]
\item Each domain is simply connected and strictly convex, and its
boundary is a real-analytic Jordan curve.
\item The two domains are Steklov isospectral:
\[
\SpecS\bigl(\Omega_{P,\varepsilon}\bigr)
=
\SpecS\bigl(\Omega_{L,\varepsilon}\bigr),
\]
including multiplicities.
\item If $i,j\in\{P,L\}$ and
$0<\varepsilon,\varepsilon'<\varepsilon_0$ satisfy
$(i,\varepsilon)\neq(j,\varepsilon')$, then
$\Omega_{i,\varepsilon}$ and $\Omega_{j,\varepsilon'}$ are not
congruent. Moreover, none of the domains admits an
orientation-reversing Euclidean symmetry.
\item Each domain has boundary length $2\pi$.
\item For each $i\in\{P,L\}$ and
$0<\varepsilon<\varepsilon_0$, there is a conformal map
$F_{i,\varepsilon}:\D\longrightarrow\Omega_{i,\varepsilon}$
with $F_{i,\varepsilon}(0)=0$ and
$F_{i,\varepsilon}'(0)>0$, extending holomorphically to a neighborhood
of $\overline\D$, such that
\[
F_{i,\varepsilon}\longrightarrow \id
\quad\text{in }C^\infty(\overline\D)
\]
as $\varepsilon\to0$.
\end{enumerate}
\end{theorem}

Thus, \cref{thm:main} gives a negative answer to the planar Steklov
analogue of Kac's question ``Can one hear the shape of a drum?''
The Steklov spectrum does not determine a bounded Euclidean plane
domain up to congruence, even within the class of simply connected,
strictly convex domains with real-analytic boundary.

\subsection{Proof strategy}

The proof combines several classical ingredients with two key
mechanisms that are new to the present construction.  We first
separate the standard input from the new parts of the argument.

The group-theoretic input is classical.  We use the degree-seven
point--line Gassmann pair for
$\operatorname{GL}(3,\Ftwo)\cong\operatorname{PSL}_2(7)$, together
with its classical $(2,3,7)$-generation; see
\cref{lem:group-input}.  The Riemann existence theorem then converts
the resulting permutation data into branched covers of the Riemann
sphere.  Likewise, the isospectral step is based on the Sunada
principle in the Steklov setting
\cite{ColboisGirouardGordonSher2024,GordonHerbrichWebb2021}, and the
analytic part of the argument uses standard facts from two-dimensional
conformal geometry, finite Blaschke products, and outer functions.

The first new ingredient is a way to implement the classical Sunada
data without introducing reflector strata, while at the same time
forcing the relevant quotients to have disk topology. From the permutation
triple we construct two degree-$n$ polynomial covers of the disk with
the same two interior branch values. The cycle structure of the permutation data is chosen so that the corresponding bordered covers have Euler characteristic one, and hence disk topology. We regard the disk as an
orbifold having cone points precisely at these branch values and pass
to a common regular orbifold cover. The corresponding intermediate
quotients are Steklov isospectral by the Steklov version of Sunada's
argument \cite[Theorem~2.3 and Remark~2.6]{GordonHerbrichWebb2021}.  The orientation of the base orbifold lifts to the common regular
cover and is preserved by all deck transformations, so the local
isotropy groups in the intermediate quotients act by rotations.
Moreover, the action is free near the boundary.  Thus the quotient
singularities are interior cone points and no reflector strata occur.  This is the point
at which the construction differs from the familiar planar Sunada
constructions, where reflector strata lead naturally to mixed
Steklov--Neumann or Steklov--Dirichlet boundary conditions rather than
to the pure Steklov problem.

Gordon, Herbrich, and Webb adapted Sunada's technique to the Steklov setting and showed that, for the planar domains arising from the classical Gordon--Webb--Wolpert construction, their argument yields isospectral sloshing problems when the boundary density is taken to vanish on the straight boundary segments corresponding to reflector edges \cite{GordonHerbrichWebb2021}.  Within this Sunada framework, passing from this mixed setting to the pure Steklov problem requires a reflector-free realization of the Sunada mechanism that still produces quotients of disk topology and can ultimately be realized as genuine Euclidean plane domains.  Thus avoiding reflectors alone is not sufficient: this must be achieved simultaneously with the topological and analytic requirements needed to recover the original planar Steklov problem.

The second new mechanism addresses a different obstruction: a weighted
Steklov problem on the disk does not automatically come from the
ordinary Steklov problem on a Euclidean plane domain.  Indeed, although
a positive boundary weight can be realized as the modulus of an outer
function, the primitive of that outer function need not be univalent.
After uniformizing the underlying quotient surfaces by the disk, the
Steklov spectra become weighted Steklov spectra
$\SpecS(\D,\rho_{j,\varepsilon})$, where
$\rho_{j,\varepsilon}$ is the boundary derivative modulus of a finite
Blaschke product.  We overcome the univalence obstruction by arranging
the two branch values to coalesce as $\varepsilon\to0$.  After a
suitable normalization, the corresponding Blaschke products converge
smoothly to the monomial $z^n$, and hence
$\rho_{j,\varepsilon}\to n$ in $C^\infty(\T)$.  For a positive
real-analytic weight sufficiently close to a constant, its normalized
outer function has positive real part, so its primitive is univalent
and realizes the weighted Steklov problem as the ordinary Steklov
problem on a Euclidean plane domain.  The resulting domains converge
smoothly, after rescaling, to the unit disk, which yields strict
convexity for sufficiently small $\varepsilon$.

Finally, we prove that the construction does not collapse under
Euclidean congruence.  A Euclidean congruence between two of the
resulting domains preserves boundary arc length and therefore forces
the boundary derivative moduli of the associated finite Blaschke
products to agree after reparametrization by the induced holomorphic
or antiholomorphic disk automorphism.  The boundary rigidity lemma
then yields a corresponding relation between the two Blaschke
products.  In the orientation-preserving case, comparison of the two
critical-value sets first forces the parameters to agree and eliminates
the residual target rotation.  If the two domains arise from different
Sunada quotients, the resulting relation is then an equivalence of the
underlying branched covers and would force the two almost-conjugate
subgroups to be conjugate.  The orientation-reversing case is excluded
by the asymmetric choice of the two branch values: their ratio is
nonreal and has modulus different from one.  This proves the pairwise
noncongruence assertion and, at the same time, excludes
orientation-reversing Euclidean symmetries.

\subsection{Organization of the paper}

In \cref{sec:analytic-preliminaries}, we collect the analytic tools
needed for the construction.  In \cref{sec:proof-main}, we establish
the group-theoretic input, construct the corresponding polynomial
covers and Sunada quotients, and complete the proof of
\cref{thm:main}.

\section{Analytic preliminaries}
\label{sec:analytic-preliminaries}

\subsection{Weighted Steklov spectra and conformal pullback}

Let $\rho$ be a positive smooth function on $\T$.  We denote by
$\SpecS(\D,\rho)$ the variational spectrum associated with the
Rayleigh quotient
\[
\mathcal R_\rho(u)
=
\frac{\displaystyle\int_\D |\nabla u|^2\dd A}
{\displaystyle\int_0^{2\pi}|u(e^{\iu t})|^2
\rho(e^{\iu t})\dd t},
\qquad
u\in H^1(\D),\quad u|_\T\not\equiv0.
\]
Equivalently, it is the spectrum of the weighted Steklov problem
\[
\Delta u=0\quad\text{in }\D,
\qquad
\partial_\nu u=\sigma\rho u\quad\text{on }\T;
\]
see, for example, \cite[Section~1]{KaoMohammadi2026}
for the corresponding weak and variational formulations.
The compactness of the trace embedding
$H^1(\D)\hookrightarrow L^2(\T)$ gives a discrete spectrum, counted
with multiplicity.

\begin{proposition}
\label{prop:conformal-pullback}
Let $F:\D\to\Omega$ be a conformal bijection onto a bounded simply
connected plane domain.  Suppose that $F$ extends to a $C^1$ diffeomorphism
$\overline\D\to\overline\Omega$ and that $F'$ does not vanish on $\T$. Then
\[
\SpecS(\Omega)=\SpecS(\D,|F'|_{\T}).
\]
\end{proposition}

\begin{proof}
For $v\in H^1(\Omega)$, set $u=v\circ F$.  Conformal invariance of
the Dirichlet integral and the boundary change of variables give
\[
\int_\Omega |\nabla v|^2\dd A
=
\int_\D |\nabla u|^2\dd A
\]
and
\[
\int_{\partial\Omega}|v|^2\dd s
=
\int_0^{2\pi}|u(e^{\iu t})|^2|F'(e^{\iu t})|\dd t.
\]
The conclusion follows from the min--max characterization; see, for instance, \cite[Theorem~7.1.9]{LMP23} and \cite[Eq.~(4.1.2)]{GirouardPolterovich2017}.
\end{proof}

\subsection{Removal of isolated interior cone points}

We use the variational definition of the Steklov spectrum on a compact
Riemannian orbisurface, as introduced in
\cite{AriasMarcoDrydenGordonHassannezhadRayStanhope2019}.
Isolated interior conical singularities may be removed by a conformal
change of metric that is trivial on the boundary, without changing the
Steklov spectrum; see \cite[Definition~2.6 and Remark~2.10]{ColboisGirouardGordonSher2024}.
In particular, this applies to interior orbifold cone points; see also
the discussion following
\cite[Corollary~8.20]{ColboisGirouardGordonSher2024}.
The next proposition records the resulting weighted-disk formulation
needed below.

\begin{proposition}
\label{prop:cone-removal}
Let $X$ be a compact orientable Riemannian orbisurface whose underlying
topological surface $|X|$ is a disk and whose singular set consists of
finitely many interior cone points. Let $\psi:\D\longrightarrow |X|^\circ$ be a conformal parametrization of the underlying Riemann surface,
extending smoothly to the boundary.  If the boundary line element of
the orbifold metric pulls back to
$\rho(e^{\iu t})\dd t$, where $\rho>0$, then
\[
\SpecS(X)=\SpecS(\D,\rho).
\]
\end{proposition}
\begin{proof}
By the conformal removal result for isolated interior conical
singularities
\cite[Remark~2.10]{ColboisGirouardGordonSher2024}, the orbifold
metric on $X$ is conformally equivalent, away from the cone points,
to a smooth metric $\widehat g$ on $|X|$, with conformal factor equal
to $1$ on the boundary, and this modification preserves the Steklov
spectrum. Hence $\SpecS(X)=\SpecS(|X|,\widehat g)$.
Pulling $\widehat g$ back by $\psi$, we obtain a smooth metric
conformal to the Euclidean metric on $\D$, with boundary line element
$\rho(e^{\iu t})\dd t$.  Conformal invariance of the Dirichlet energy
in dimension two and the min--max characterization therefore give $\SpecS(|X|,\widehat g)=\SpecS(\D,\rho)$.
\end{proof}

\subsection{Realization of weights near a constant}

The next proposition is the analytic mechanism that turns the
weighted spectra produced below into spectra of Euclidean domains.

\begin{proposition}
\label{prop:outer-realization}
Let $c>0$.  Every positive real-analytic function $\rho$ in a
sufficiently small $C^\infty(\T)$-neighborhood of the constant $c$ is
realized by a holomorphic map $F_\rho$ on a neighborhood of
$\overline\D$ with the following properties:
\begin{enumerate}[label=\textup{(\roman*)}]
\item $F_\rho(0)=0$, $F_\rho'(0)>0$, and $F_\rho$ is injective on
$\overline\D$;
\item $|F_\rho'|=\rho$ on $\T$.
\end{enumerate}
Thus $F_\rho(\D)$ is a bounded simply connected domain with
real-analytic Jordan boundary and
\[
\SpecS(F_\rho(\D))=\SpecS(\D,\rho).
\]
Moreover, if $\rho_k\to c$ in $C^\infty(\T)$, then the corresponding
maps may be chosen so that
\[
F_{\rho_k}\longrightarrow c\id
\quad\text{in }C^\infty(\overline\D).
\]
\end{proposition}

\begin{proof}
For a positive real-analytic $\rho$, define its normalized outer
function by
\[
O_\rho(z)
=
\exp\left\{
\frac1{2\pi}
\int_0^{2\pi}
\frac{e^{\iu t}+z}{e^{\iu t}-z}
\log\rho(e^{\iu t})\dd t
\right\}.
\]
The outer-function theorem gives $|O_\rho|=\rho$ on $\T$ and
$O_\rho(0)>0$; see \cite{Duren1970}.  Real analyticity of
$\log\rho$ implies that $O_\rho$ extends holomorphically, without
zeros, across $\T$.  The map $\rho\mapsto O_\rho$ is continuous at
the constant $c$ in the smooth topology, because harmonic conjugation
is continuous on smooth functions of mean zero. Hence
$O_{\rho_k}\longrightarrow c$ in $C^\infty(\overline\D)$ whenever
$\rho_k\to c$. In particular,
$\operatorname{Re}O_\rho>0$ on $\overline\D$ whenever $\rho$ is
sufficiently close to $c$.

Set $F_\rho(z)=\int_0^z O_\rho(\zeta)\dd\zeta$. For distinct $z_1,z_2\in\overline\D$, convexity of $\overline\D$
yields
\[
\frac{F_\rho(z_2)-F_\rho(z_1)}{z_2-z_1}
=
\int_0^1O_\rho\bigl(z_1+s(z_2-z_1)\bigr)\dd s.
\]
The real part of the right-hand side is positive.  Hence $F_\rho$ is
injective on $\overline\D$.  The remaining assertions follow from
the construction, the smooth convergence above, and
\cref{prop:conformal-pullback}.
\end{proof}

\subsection{Boundary rigidity for finite Blaschke products}

The following elementary observation will later convert congruence of
the Euclidean domains into equivalence of their branched covers.

\begin{lemma}
\label{lem:blaschke-boundary-rigidity}
Let $B_1$ and $B_2$ be finite Blaschke products of the same degree.
\begin{enumerate}[label=\textup{(\roman*)}]
\item If $\varphi\in\Aut(\D)$ and
\[
\left|\frac{d}{dt}B_1(e^{\iu t})\right|
=
\left|\frac{d}{dt}B_2(\varphi(e^{\iu t}))\right|
\qquad(t\in\R),
\]
then $B_1=\eta B_2\circ\varphi$ for some $\eta\in\T$.
\item If $\varphi$ is an antiholomorphic automorphism of $\D$ and the
same equality of boundary derivative moduli holds, then $B_1(z)=\eta\,\overline{B_2(\varphi(z))}$ for some $\eta\in\T$.
\end{enumerate}
\end{lemma}

\begin{proof}
For \textup{(i)}, choose increasing lifts $b_1,b_2:\R\to\R$ such that
$B_1(e^{\iu t})=e^{\iu b_1(t)}$ and
$B_2(\varphi(e^{\iu t}))=e^{\iu b_2(t)}$.
Finite Blaschke products and holomorphic disk automorphisms preserve
boundary orientation.  The hypothesis therefore says $b_1'=b_2'$,
so $b_1-b_2$ is constant.  The boundary maps differ by a rotation,
and the identity extends holomorphically to $\D$.

For \textup{(ii)}, the function $C(z)=\overline{B_2(\varphi(z))}$ is a finite Blaschke product of the same degree as $B_2$.  Its
derivative modulus on the circle is the boundary derivative modulus
of $B_2\circ\varphi$.  Part \textup{(i)} applied to $B_1$ and $C$
gives the conclusion.
\end{proof}

\section{Proof of the main theorem}
\label{sec:proof-main}

\subsection{Group-theoretic input}

If a finite group $G$ acts on a finite set $S$ and $g\in G$, let
$c_S(g)$ denote the number of cycles of the permutation induced by
$g$ on $S$, including one-cycles.

Recall that subgroups $H_1,H_2\leq G$ are \emph{almost conjugate} if
every conjugacy class of $G$ meets them in the same number of
elements.  Equivalently, the permutation representations on
$G/H_1$ and $G/H_2$ have the same character. In particular, the two actions have the
same cycle type for each $g\in G$: the numbers of fixed points of all
powers $g^k$ determine the numbers of cycles of every length by
M\"obius inversion.

\begin{lemma}
\label{lem:group-input}
There exist a finite group $G$, nonconjugate almost-conjugate
subgroups $H_1,H_2\leq G$ of a common index $n$, and elements
$x,y\in G$ such that
\begin{enumerate}[label=\textup{(\alph*)}]
\item $x$ and $y$ generate $G$;
\item $xy$ acts as an $n$-cycle on $G/H_1$;
\item writing $c(g)=c_{G/H_1}(g)$, one has
\[
c(x)+c(y)=n+1,
\qquad
c(x)<n,
\qquad
c(y)<n.
\]
\end{enumerate}
In fact, one may take $n=7$,
$G=\operatorname{GL}(3,\Ftwo)$, with
$\operatorname{ord}(x)=2$, $\operatorname{ord}(y)=3$, and
$\operatorname{ord}(xy)=7$.
\end{lemma}

\begin{proof}
Let
\[
G=\operatorname{GL}(3,\Ftwo),\qquad
W=\Ftwo^3,\qquad
\mathcal P=W\setminus\{0\},\qquad
\mathcal L=W^*\setminus\{0\}.
\]
Let $H_P$ be the stabilizer of a nonzero vector and $H_L$ the
stabilizer of a nonzero covector, where $G$ acts on $W^*$ by the
contragredient action.  This is the classical degree-seven
point--line Gassmann pair used in the Gordon--Webb--Wolpert
construction; see \cite[Sec.~3.3]{GordonWebbWolpert1992}.

For completeness, the two permutation characters agree. Indeed,
for $g\in G$ the numbers of fixed points on $\mathcal P$ and
$\mathcal L$ are $2^{\dim\ker(g-I)}-1$ and $2^{\dim\ker((g^{-1})^T-I)}-1$, respectively, and these are equal because $(g^{-1})^T-I=(g^{-1})^T(I-g^T)$. Thus $H_P$ and $H_L$ are almost conjugate, and both have index seven.

They are not conjugate.  To see this, take $H_L$ to be the stabilizer
of $e_1^T$.  Thus its elements have first row $(1,0,0)$.  Varying the
lower-left column while keeping the lower-right $2\times2$ block
equal to the identity shows that a vector fixed by all of $H_L$ must
have first coordinate zero.  Varying the lower-right block through
$\operatorname{GL}(2,\Ftwo)$ then shows that no nonzero vector is
fixed by all of $H_L$.  Every conjugate of $H_P$, on the other hand,
fixes a nonzero vector.  Hence $H_P$ and $H_L$ are not conjugate.

Since $\Ftwo^\times=\{1\}$, one has
\[
\operatorname{GL}(3,\Ftwo)
=
\operatorname{SL}(3,\Ftwo)
=
\operatorname{PSL}(3,\Ftwo)
\cong\operatorname{PSL}_2(7).
\]
The latter is a Hurwitz group, so there exist $x,y\in G$ generating
$G$ with $\operatorname{ord}(x)=2$, $\operatorname{ord}(y)=3$, and
$\operatorname{ord}(xy)=7$; see
\cite{LucchiniTamburiniWilson2000}.

Consider the action on $\mathcal P$.  Since the characteristic is
two, for the involution $x$ the nonzero operator $N=x-I$ satisfies
$N^2=0$.  Hence $\operatorname{rank}N=1$, so
$\dim\ker(x-I)=2$.  Thus $x$ fixes three nonzero vectors and has
cycle type $2^2 1^3$, so $c(x)=5$.

An element of order three in $\operatorname{GL}(3,\Ftwo)$ has a
one-dimensional fixed space and an irreducible two-dimensional block
with minimal polynomial $t^2+t+1$.  Hence $y$ has cycle type $3^2 1$, so $c(y)=3$.

Finally, the action of $G$ on $G/H_{P}$ is faithful. Choose $v_0\in\mathcal P$ and let $H_P=\operatorname{Stab}_G(v_0)$. Since $G$ acts transitively on $\mathcal P$, the map $G/H_P\to\mathcal P$, $gH_P\mapsto gv_0$, is a $G$-equivariant bijection. Hence the coset action of $G$ on $G/H_P$ may be identified with the natural action of $G$ on $\mathcal P$. Therefore
$xy$, which has order seven, induces a permutation of order seven on
the seven-element set $\mathcal P$, and hence is a seven-cycle.
Consequently
\[
c(x)+c(y)=5+3=8=7+1.
\]
Taking $H_1=H_P$ and $H_2=H_L$ proves the lemma.
\end{proof}

From now on, fix $G,H_1,H_2,x,y$ and $n$ as in
\cref{lem:group-input}.

\subsection{Polynomial covers from permutation data}

Let $\pi_j$ be the action of $G$ on $G/H_j$.  Almost conjugacy implies
that the hypotheses involving cycle data hold for both $\pi_1$ and
$\pi_2$.  Consider the transitive permutation triples
\[
\bigl(\pi_j(x),\pi_j(y),\pi_j((xy)^{-1})\bigr),
\qquad j\in\{1,2\}.
\]
Each triple has product one.  By the Riemann existence theorem
\cite{Forster1981}, it is the monodromy triple of a connected
degree-$n$ branched cover of the Riemann sphere
$f_j:\Sigma_j\longrightarrow\widehat{\C}$, where
$\widehat{\C}=\C\cup\{\infty\}$, with branch values $0$, $1$, and
$\infty$. Its total ramification is
\[
(n-c(x))+(n-c(y))+(n-1)=2n-2.
\]
The Riemann--Hurwitz formula therefore shows that $\Sigma_j$ has genus
zero.  The monodromy over infinity is an $n$-cycle, so
$f_j^{-1}(\infty)$ consists of a single point of local degree $n$.
After identifying $\Sigma_j$ with $\widehat{\C}$, sending that point
to infinity, and rescaling the affine source coordinate, the covering
is represented by a monic polynomial
\begin{equation}
p_j(z)=z^n+\sum_{k=0}^{n-1}a_{j,k}z^k.
\label{eq:general-polynomials}
\end{equation}
Thus the only finite critical values of $p_j$ are $0$ and $1$, with
local monodromies $\pi_j(x)$ and $\pi_j(y)$.

Fix once and for all a complex number
$\lambda_*\in\C\setminus\R$ with $|\lambda_*|>1$; for definiteness,
we take $\lambda_*=2+\iu$. Set
\[
u_\varepsilon=\varepsilon,
\qquad
v_\varepsilon=\lambda_*\varepsilon,
\qquad
\delta_\varepsilon=v_\varepsilon-u_\varepsilon
=(\lambda_*-1)\varepsilon,
\]
and define
\begin{equation}
q_{j,\varepsilon}(z)
=u_\varepsilon+\delta_\varepsilon p_j(z),
\qquad
Y_{j,\varepsilon}=q_{j,\varepsilon}^{-1}(\D).
\label{eq:general-rescaled-covers}
\end{equation}
For sufficiently small $\varepsilon>0$, both finite critical values
$u_\varepsilon$ and $v_\varepsilon$ lie in $\D$.

\begin{lemma}
\label{lem:general-disk-covers}
For all sufficiently small $\varepsilon$, the set
$Y_{j,\varepsilon}$ is a bounded simply connected domain with
real-analytic Jordan boundary, and $q_{j,\varepsilon}:Y_{j,\varepsilon}\longrightarrow\D$ is a proper holomorphic map of degree $n$.  Its two critical values
are $u_\varepsilon$ and $v_\varepsilon$.
\end{lemma}

\begin{proof}
After removing the two branch values and their preimages, the
monodromy group is generated by $\pi_j(x)$ and $\pi_j(y)$.  It is
transitive because $x$ and $y$ generate $G$.  Hence
$q_{j,\varepsilon}^{-1}(\overline\D)$ is connected.  No critical
value lies on $\T$, so this inverse image is a compact orientable
surface with real-analytic boundary.  The Riemann--Hurwitz formula for branched coverings of compact
surfaces with boundary gives
\[
\chi\bigl(q_{j,\varepsilon}^{-1}(\overline\D)\bigr)
=n-(n-c(x))-(n-c(y))=1.
\]
A connected orientable compact surface with nonempty boundary and
Euler characteristic one has genus zero and one boundary component.
Its interior is therefore a disk.  Boundedness follows because
$q_{j,\varepsilon}$ is a polynomial.  Properness and the degree
statement are immediate. The inequalities $c(x)<n$ and $c(y)<n$ in
\cref{lem:group-input} ensure that both finite branch values are
genuine critical values.
\end{proof}

\subsection{Sunada quotients and weighted disks}

Set $r:=\operatorname{ord}(x)$ and
$s:=\operatorname{ord}(y)$.
Let $\mathcal O_\varepsilon$ be the compact disk orbifold whose
underlying space is $\overline\D$, with an interior cone point of order
$r$ at $u_\varepsilon$ and an interior cone point of order $s$ at
$v_\varepsilon$.  Its orbifold fundamental group is
\[
\Gamma_\varepsilon
=
\pi_1^{\mathrm{orb}}(\mathcal O_\varepsilon)
=
\langle \alpha,\beta:\alpha^r=\beta^s=1\rangle
\simeq C_r*C_s,
\]
where $C_m$ denotes the cyclic group of order $m$.
The assignment $\alpha\mapsto x$, $\beta\mapsto y$ defines an
epimorphism $\Phi:\Gamma_\varepsilon\longrightarrow G$.
By the definitions of $r$ and $s$, the restrictions of $\Phi$ to the
factors $C_r$ and $C_s$ are injective.  Every finite-order element of
a free product is conjugate into one of its factors
\cite{LyndonSchupp1977}.  Hence the kernel of $\Phi$ is torsion-free.
The regular orbifold cover $M_\varepsilon\longrightarrow\mathcal O_\varepsilon$ associated with $\ker\Phi$ is consequently a smooth compact surface
with boundary, with deck group $G$.  The action of $G$ is free on a
collar of the boundary: the covering is ordinary there, and a deck
transformation fixing a boundary point must be the identity.  In
particular, none of the quotients below has reflector strata on its
boundary.

Choose the generators $\alpha$ and $\beta$ above as positively oriented
meridians about $u_\varepsilon$ and $v_\varepsilon$, compatibly with
the monodromy triples defining $q_{j,\varepsilon}$.  Identify a fiber
of the regular cover with $G$ so that monodromy acts by left
multiplication through $\Phi$ and deck transformations act on the
right.  For $j\in\{1,2\}$, let $X_{j,\varepsilon}=M_\varepsilon/H_j$, where the quotient is taken with respect to the restricted right
action.  The quotient maps form the diagram
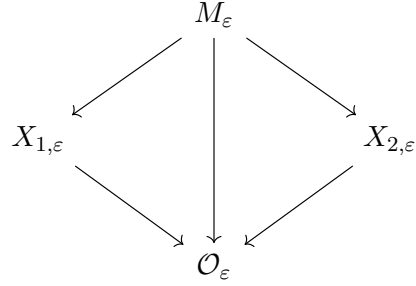
\[
\begin{tikzcd}[column sep=large,row sep=large]
& M_\varepsilon \arrow[dl] \arrow[dr] \arrow[dd] & \\
X_{1,\varepsilon} \arrow[dr] & & X_{2,\varepsilon} \arrow[dl] \\
& \mathcal O_\varepsilon &
\end{tikzcd}
\]
With this convention, the fiber of
$X_{j,\varepsilon}\to\mathcal O_\varepsilon$ is $G/H_j$ and its
monodromy is $\pi_j\circ\Phi$.  Hence the induced branched cover
$|X_{j,\varepsilon}|\to\overline\D$ has the same monodromy data as
$q_{j,\varepsilon}:\overline{Y_{j,\varepsilon}}\to\overline\D$.
Since the boundary monodromy is the $n$-cycle
$\pi_j((xy)^{-1})$, both covers have a single boundary component.
By the uniqueness statement for branched covers with prescribed
monodromy (equivalently, by applying the uniqueness part of the
Riemann existence theorem after filling the boundary), there is a
conformal equivalence of bordered surfaces
$|X_{j,\varepsilon}|\to\overline{Y_{j,\varepsilon}}$ over
$\overline\D$; see \cite{Forster1981}. We henceforth identify these bordered surfaces
via such an equivalence.  The orientation of
$\mathcal O_\varepsilon$ lifts to $M_\varepsilon$, and every deck
transformation preserves the lifted orientation.  Hence every local
isotropy group in $X_{j,\varepsilon}$ acts by rotations.  Since the
deck action is free near the boundary, $|X_{j,\varepsilon}|$ is a
disk, its orbifold singularities are all interior cone points, and its
boundary is smooth.

Choose a smooth orbifold metric $g_\varepsilon$ compatible with the
orbifold conformal structure on $\mathcal O_\varepsilon$ induced by the
standard complex structure of $\D$, and equal to the Euclidean metric
$|dw|^2$ on a collar of $\T$.  Pull it back to
$M_\varepsilon$ and descend it to the intermediate quotients.

\begin{proposition}
\label{prop:general-sunada}
The orbisurfaces $X_{1,\varepsilon}$ and $X_{2,\varepsilon}$ have
identical Steklov spectra, including multiplicities.
\end{proposition}

\begin{proof}
This is the orbifold version of the Steklov Sunada theorem; see
\cite[Theorem~2.3 and Remark~2.6]{GordonHerbrichWebb2021} and
\cite[Theorem~9.9]{ColboisGirouardGordonSher2024}. Let $\Lambda_M$ denote the Dirichlet-to-Neumann operator on $\partial M_\varepsilon$. Since the deck transformations act on $M_\varepsilon$ by isometries preserving the boundary, they induce a left action of $G$ on $C^\infty(\partial M_\varepsilon)$ by $(g\cdot f)(p)=f(p\cdot g)$. If $u_f$ denotes the harmonic extension of $f$, then $g\cdot u_f$ is the harmonic extension of $g\cdot f$. Since the action preserves the outward unit normal, $\Lambda_M(g\cdot f)=g\cdot\Lambda_M f$. Thus $\Lambda_M$ commutes with the $G$-action.

The operator $\Lambda_M$ is a nonnegative self-adjoint elliptic pseudodifferential operator of order one on the compact manifold $\partial M_\varepsilon$, and hence has discrete spectrum with finite-dimensional eigenspaces. Therefore, for every Steklov eigenvalue $\sigma$, the eigenspace $E_\sigma:=\ker(\Lambda_M-\sigma I)$ is a finite-dimensional $G$-module.

For a subgroup $H\le G$, let $E_\sigma^H:=\{f\in E_\sigma:h\cdot f=f\text{ for every }h\in H\}$. Let $R_\sigma:G\to\operatorname{GL}(E_\sigma)$ be the corresponding
representation, and let $\chi_\sigma$ be its character.  The averaging
operator
$P_H:=|H|^{-1}\sum_{h\in H}R_\sigma(h)$
is the projection of $E_\sigma$ onto $E_\sigma^H$. Hence $\dim E_\sigma^H=\operatorname{tr}P_H=|H|^{-1}\sum_{h\in H}\chi_\sigma(h)$. Since $H_1$ and $H_2$ are almost conjugate, one has $|H_1\cap C|=|H_2\cap C|$ for every conjugacy class $C$ of $G$, and in particular $|H_1|=|H_2|$. Since $\chi_\sigma$ is constant on conjugacy classes,
\[
\dim E_\sigma^{H_1}
=\frac{1}{|H_1|}\sum_C |H_1\cap C|\,\chi_\sigma(C)
=\frac{1}{|H_2|}\sum_C |H_2\cap C|\,\chi_\sigma(C)
=\dim E_\sigma^{H_2}.
\]

It remains to identify these fixed subspaces with the Steklov eigenspaces on the quotients. Let $r_j:M_\varepsilon\to X_{j,\varepsilon}=M_\varepsilon/H_j$ be the quotient map. Pullback identifies smooth functions on $\partial X_{j,\varepsilon}$ with the $H_j$-invariant smooth functions on $\partial M_\varepsilon$. If $f\in C^\infty(\partial X_{j,\varepsilon})$ and $u_f$ is its harmonic extension to $X_{j,\varepsilon}$, then $u_f\circ r_j$ is the harmonic extension of $r_j^*f$ to $M_\varepsilon$. Since $r_j$ is an orbifold local isometry, and an ordinary local isometry near the boundary, one has $\Lambda_M r_j^*=r_j^*\Lambda_{X_{j,\varepsilon}}$.

Conversely, if $f\in E_\sigma^{H_j}$, then for every $h\in H_j$ the harmonic functions $h\cdot u_f$ and $u_f$ have the same boundary value, because $h\cdot f=f$. By uniqueness of the harmonic extension, $h\cdot u_f=u_f$. Hence $u_f$ is $H_j$-invariant and descends to a harmonic orbifold function on $X_{j,\varepsilon}$. Therefore pullback induces an isomorphism $\ker(\Lambda_{X_{j,\varepsilon}}-\sigma I)\cong E_\sigma^{H_j}$, and consequently $\operatorname{mult}_{X_{j,\varepsilon}}(\sigma)=\dim E_\sigma^{H_j}$. Since $\dim E_\sigma^{H_1}=\dim E_\sigma^{H_2}$ for every $\sigma$, the orbisurfaces $X_{1,\varepsilon}$ and $X_{2,\varepsilon}$ have identical Steklov spectra, including multiplicities.
\end{proof}

By \cref{lem:general-disk-covers}, each $Y_{j,\varepsilon}$ is a bounded
simply connected domain with real-analytic Jordan boundary, and
$q_{j,\varepsilon}:Y_{j,\varepsilon}\to\D$ is a proper holomorphic map
of degree $n$.
Choose Riemann maps
\[
\psi_{j,\varepsilon}:\D\longrightarrow Y_{j,\varepsilon}.
\]
Because $\partial Y_{j,\varepsilon}$ is real analytic, each map
extends biholomorphically across $\T$.  Define
\begin{equation}
B_{j,\varepsilon}
=q_{j,\varepsilon}\circ\psi_{j,\varepsilon},
\qquad
\rho_{j,\varepsilon}(e^{\iu t})
=|B_{j,\varepsilon}'(e^{\iu t})|.
\label{eq:general-blaschke-weights}
\end{equation}
The map $B_{j,\varepsilon}$ is a finite Blaschke product of degree
$n$.  Since the base metric is Euclidean near $\T$, the boundary line
element of $X_{j,\varepsilon}$ pulls back under
$\psi_{j,\varepsilon}$ to
$\rho_{j,\varepsilon}(e^{\iu t})\dd t$.  Applying
\cref{prop:cone-removal,prop:general-sunada} gives
\begin{equation}
\SpecS(\D,\rho_{1,\varepsilon})
=
\SpecS(\D,\rho_{2,\varepsilon}).
\label{eq:general-weighted-isospectrality}
\end{equation}

\subsection{Coalescence of the branch values}

We now normalize the Riemann maps in
\cref{eq:general-blaschke-weights} so that the weights approach a
constant.

\begin{proposition}
\label{prop:general-coalescence}
The maps $\psi_{j,\varepsilon}$ may be chosen so that
\begin{equation}
\rho_{j,\varepsilon}\longrightarrow n
\quad\text{in }C^\infty(\T)
\label{eq:general-weight-convergence}
\end{equation}
as $\varepsilon\to0$.
\end{proposition}

\begin{proof}
Choose $\tau_\varepsilon\in\C$ with
$\tau_\varepsilon^n=\delta_\varepsilon$, and set $Q_{j,\varepsilon}(z)
=q_{j,\varepsilon}(\tau_\varepsilon^{-1}z)$.
Using \cref{eq:general-polynomials,eq:general-rescaled-covers}, we
find
\[
Q_{j,\varepsilon}(z)
=u_\varepsilon+z^n
+
\sum_{k=0}^{n-1}a_{j,k}
\delta_\varepsilon\tau_\varepsilon^{-k}z^k.
\]
For $0\leq k\leq n-1$, we have \(|\delta_\varepsilon\tau_\varepsilon^{-k}|=|\delta_\varepsilon|^{1-k/n}\longrightarrow0\). Consequently,
\begin{equation}
Q_{j,\varepsilon}\longrightarrow z^n
\quad\text{locally uniformly on }\C,
\text{ together with all derivatives.}
\label{eq:general-polynomial-convergence}
\end{equation}

Put
\[
U_{j,\varepsilon}
=\tau_\varepsilon Y_{j,\varepsilon}
=\{z\in\C:|Q_{j,\varepsilon}(z)|<1\}.
\]
For small $\varepsilon$, this is a simply connected domain containing
$0$.  We claim that its Carath\'eodory kernel with respect to $0$ is
$\D$.  Every compact subset of $\D$ is eventually contained in
$U_{j,\varepsilon}$ by
\cref{eq:general-polynomial-convergence}.  If $|z|>1$, then
$|Q_{j,\varepsilon}(z)|>1$ for all sufficiently small
$\varepsilon$, so no such point belongs to the kernel.  Finally, for
every fixed $R_0>1$, uniform convergence on $|z|=R_0$ gives
$|Q_{j,\varepsilon}|>1$ on that circle for small $\varepsilon$.
Since $U_{j,\varepsilon}$ is connected and contains $0$, it is
contained in $\{|z|<R_0\}$.  This proves the claim and supplies the
needed uniform boundedness.

Let $\phi_{j,\varepsilon}:\D\longrightarrow U_{j,\varepsilon}$ be normalized by $\phi_{j,\varepsilon}(0)=0$ and
$\phi_{j,\varepsilon}'(0)>0$.  The Carath\'eodory kernel theorem~\cite[Theorem~1.8]{Pommerenke1992} gives
$\phi_{j,\varepsilon}\longrightarrow\id$ locally uniformly on $\D$.
Take
$\psi_{j,\varepsilon}=\tau_\varepsilon^{-1}\phi_{j,\varepsilon}$.
Then
\[
B_{j,\varepsilon}
=Q_{j,\varepsilon}\circ\phi_{j,\varepsilon}
\longrightarrow z^n
\]
locally uniformly on $\D$.

Write
\[
B_{j,\varepsilon}(z)
=\eta_{j,\varepsilon}
\prod_{\ell=1}^n
\frac{z-\alpha_{j,\ell,\varepsilon}}
{1-\overline{\alpha_{j,\ell,\varepsilon}}z},
\qquad |\eta_{j,\varepsilon}|=1.
\]
For every $0<r_0<1$, Rouch\'e's theorem on $|z|=r_0$ shows that all $n$
zeros lie in $|z|<r_0$ when $\varepsilon$ is small.  Hence $\max_\ell|\alpha_{j,\ell,\varepsilon}|\longrightarrow0$. Evaluation at a fixed nonzero point of $\D$ gives
$\eta_{j,\varepsilon}\to1$.  The poles of the factors are therefore
outside a fixed neighborhood of $\overline\D$ for all sufficiently
small $\varepsilon$, and the products converge there holomorphically
with all derivatives.  Thus
$B_{j,\varepsilon}\to z^n$ in $C^\infty(\overline\D)$.
Differentiating on $\T$ proves
\cref{eq:general-weight-convergence}.
\end{proof}

\subsection{Euclidean realization and noncongruence}

We now complete the construction.  By
\cref{prop:general-coalescence,prop:outer-realization}, after
decreasing $\varepsilon_0$ if necessary, there are conformal embeddings
$F_{j,\varepsilon}:\overline\D\longrightarrow\C$ that extend
holomorphically to a neighborhood of $\overline\D$,
satisfy
\begin{equation}
F_{j,\varepsilon}(0)=0,
\qquad
F_{j,\varepsilon}'(0)>0,
\qquad
|F_{j,\varepsilon}'|=\rho_{j,\varepsilon}
\ \text{on }\T,
\label{eq:general-realizing-maps}
\end{equation}
and obey
\begin{equation}
F_{j,\varepsilon}\longrightarrow n\id
\quad\text{in }C^\infty(\overline\D).
\label{eq:general-map-convergence}
\end{equation}
Set $\Omega_{j,\varepsilon}=F_{j,\varepsilon}(\D)$.  Combining
\cref{eq:general-weighted-isospectrality,prop:conformal-pullback}
gives
\begin{equation}
\SpecS(\Omega_{1,\varepsilon})
=
\SpecS(\Omega_{2,\varepsilon})
\label{eq:general-planar-isospectrality}
\end{equation}
with multiplicities.

We next prove the pairwise noncongruence statement. Suppose that a Euclidean rigid motion $J$ maps $\Omega_{j,\varepsilon}$ onto $\Omega_{k,\varepsilon'}$.  If $J$ preserves orientation, the induced map $\varphi=F_{k,\varepsilon'}^{-1}\circ J\circ F_{j,\varepsilon}$ is a holomorphic automorphism of $\D$.  Preservation of boundary arc length and \cref{eq:general-blaschke-weights,eq:general-realizing-maps} imply $|B_{j,\varepsilon}'|=|(B_{k,\varepsilon'}\circ\varphi)'|$ on $\T$. By \cref{lem:blaschke-boundary-rigidity},
\begin{equation}
B_{j,\varepsilon}
=\eta B_{k,\varepsilon'}\circ\varphi
\label{eq:general-cover-equivalence}
\end{equation}
for some $\eta\in\T$. The critical values in $\D$ of the left-hand side are $\{\varepsilon,\lambda_*\varepsilon\}$, whereas the critical values in $\D$ of the right-hand side are $\{\eta\varepsilon',\eta\lambda_*\varepsilon'\}$. Since $|\lambda_*|>1$, each set has a unique element of smaller
modulus.  Equality of the two sets therefore gives
$\varepsilon=\varepsilon'$ and $\eta=1$.

If $j\neq k$, define
\[
h=\psi_{k,\varepsilon}\circ\varphi
\circ\psi_{j,\varepsilon}^{-1}
:Y_{j,\varepsilon}\longrightarrow Y_{k,\varepsilon}.
\]
Since $\eta=1$, \cref{eq:general-cover-equivalence} becomes
$q_{j,\varepsilon}=q_{k,\varepsilon}\circ h$; it is therefore an
equivalence of the two branched covers over the identity of $\D$.
Their monodromy actions are consequently simultaneously conjugate.
Because $x$ and $y$ generate $G$, this gives an isomorphism of
transitive $G$-sets $G/H_j\simeq G/H_k$, which is equivalent to
conjugacy of $H_j$ and $H_k$.  Since $j\neq k$, this contradicts
the nonconjugacy of $H_1$ and $H_2$ in \cref{lem:group-input}.
Thus an orientation-preserving congruence is possible only when
$(j,\varepsilon)=(k,\varepsilon')$.

If $J$ reverses orientation, the induced $\varphi$ is
antiholomorphic, and \cref{lem:blaschke-boundary-rigidity} gives
\[
B_{j,\varepsilon}(z)
=\eta\,\overline{B_{k,\varepsilon'}(\varphi(z))}.
\]
The critical values in $\D$ of the right-hand side are
\[
\{\eta\varepsilon',
\eta\overline{\lambda_*}\varepsilon'\}.
\]
Comparison with the critical values in $\D$ of the left-hand side
again gives $\varepsilon=\varepsilon'$ and $\eta=1$ by matching
the unique values of smaller modulus.  The remaining equality would
require $\lambda_*=\overline{\lambda_*}$, contrary to
$\lambda_*=2+\iu$. Hence no orientation-reversing congruence
exists, even from one domain to itself.

For a finite Blaschke product of degree $n$, the argument of its
boundary values increases by $2\pi n$.  Therefore
\cref{eq:general-blaschke-weights,eq:general-realizing-maps} give
\begin{equation}
|\partial\Omega_{j,\varepsilon}|
=
\int_0^{2\pi}|B_{j,\varepsilon}'(e^{\iu t})|\dd t
=2\pi n.
\label{eq:general-perimeter}
\end{equation}
We now normalize the scale.  Define
\[
\widehat F_{j,\varepsilon}
=\frac{1}{n}F_{j,\varepsilon},
\qquad
\widehat\Omega_{j,\varepsilon}
=\frac{1}{n}\Omega_{j,\varepsilon}.
\]
A common dilation multiplies every Steklov eigenvalue in both spectra
by the same reciprocal factor, so
\cref{eq:general-planar-isospectrality} is preserved.  Pairwise
noncongruence is also preserved.  The relations in
\cref{eq:general-map-convergence,eq:general-perimeter} give
\[
\widehat F_{j,\varepsilon}\longrightarrow\id
\quad\text{in }C^\infty(\overline\D),
\qquad
|\partial\widehat\Omega_{j,\varepsilon}|=2\pi.
\]

It remains to verify strict convexity.  The boundary parametrizations
$\widehat\gamma_{j,\varepsilon}(t)
=\widehat F_{j,\varepsilon}(e^{\iu t})$
converge in $C^\infty$ to the positively oriented unit circle
$e^{\iu t}$.  Their signed curvatures therefore converge uniformly
to $1$ and are positive for small $\varepsilon$.  Each boundary
consequently encloses a strictly convex domain.

Set
$\Omega_{P,\varepsilon}:=\widehat\Omega_{1,\varepsilon}$ and
$\Omega_{L,\varepsilon}:=\widehat\Omega_{2,\varepsilon}$, and set
$F_{P,\varepsilon}:=\widehat F_{1,\varepsilon}$ and
$F_{L,\varepsilon}:=\widehat F_{2,\varepsilon}$.
All the assertions of \cref{thm:main} now follow.

\section*{Acknowledgments}
The authors would like to express their sincere gratitude to Zuoqin Wang for his valuable comments, which have significantly improved an earlier version of this paper.

\medskip

{\noindent\bf Declaration of AI usage.}
The initial idea of adapting the Sunada construction of
Gordon, Webb, and Wolpert \cite{GordonWebbWolpert1992} to the planar
Steklov problem was proposed by the authors.  Starting from this
strategy, ChatGPT provided substantial assistance in
developing the concrete counterexample construction, including the
passage from the orbifold construction to weighted Steklov problems
on the disk and their subsequent realization by Euclidean plane
domains.  It also assisted with several technical arguments, including
\cref{prop:outer-realization,lem:blaschke-boundary-rigidity,prop:general-coalescence}. The authors subsequently checked, revised,
and rewrote the arguments and take full responsibility for all
mathematical claims in the paper.

\end{document}